\documentclass[11pt]{article}
\usepackage{amsmath, amssymb, amsthm}

\date{}
\title{\vspace{-1cm} Sperner's Theorem and a Problem of Erd\H{o}s-Katona-Kleitman}
\author{
Shagnik Das \thanks{Department of Mathematics, UCLA, Los Angeles, CA, 90095. Email: shagnik@ucla.edu.}
\and
Wenying Gan \thanks{Department of Mathematics, UCLA, Los
Angeles, CA, 90095. Email: wgan@math.ucla.edu.}
\and
Benny Sudakov \thanks{Department of Mathematics, UCLA, Los Angeles, CA 90095.
Email: bsudakov@math.ucla.edu. Research supported in part by NSF
grant DMS-1101185, by AFOSR MURI grant FA9550-10-1-0569 and by a USA-Israel BSF grant.}
}

\theoremstyle{plain}
\newtheorem{THM}{Theorem}[section]
\newtheorem*{THM*}{Theorem}
\newtheorem{PROP}[THM]{Proposition}
\newtheorem{LEMMA}[THM]{Lemma}

\newtheorem{CONJ}[THM]{Conjecture}

\theoremstyle{definition}

\newcommand{\floor}[1]{\left\lfloor #1 \right\rfloor}
\newcommand{\ceil}[1]{\left\lceil #1 \right\rceil}
\newcommand{\card}[1]{\left| #1 \right|}
\newcommand{\cA}{\mathcal{A}}
\newcommand{\cB}{\mathcal{B}}
\newcommand{\cC}{\mathcal{C}}
\newcommand{\cF}{\mathcal{F}}
\newcommand{\cG}{\mathcal{G}}
\newcommand{\cH}{\mathcal{H}}
\newcommand{\sperner}{\binom{n}{\floor{ n / 2} }}

\begin{document}
\maketitle

\begin{abstract}
A central result in extremal set theory is the celebrated theorem of Sperner from 1928, which gives the size of the largest family of subsets of $[n]$ not containing a $2$-chain $F_1 \subset F_2$.  Erd\H{o}s extended this theorem to determine the largest family without a $k$-chain $F_1 \subset F_2 \subset \hdots \subset F_k$.  Erd\H{o}s and Katona, followed by Kleitman, asked how many chains must appear in families with sizes larger than the corresponding extremal bounds.

In 1966, Kleitman resolved this question for $2$-chains, showing that the number of such chains is minimized by taking sets as close to the middle level as possible.  Moreover, he conjectured the extremal families were the same for $k$-chains, for all $k$.  In this paper, making the first progress on this problem,
we verify Kleitman's conjecture for the families whose size is at most the size of the $k+1$ middle levels. We also characterize all extremal configurations.
\end{abstract}

\section{Introduction} \label{sec:introduction}

Sperner's Theorem is a central result in extremal set theory, giving the size of the largest family of sets not containing a $2$-chain $F_1 \subset F_2$.  Erd\H{o}s later extended this theorem to determine the largest family without a $k$-chain $F_1 \subset F_2 \subset \hdots \subset F_k$.  A natural question is to ask how many $k$-chains must appear in a family larger than this extremal bound.

More precisely, we consider the following problem, first posed by Erd\H{o}s and Katona and then extended by Kleitman some fifty years ago.  Given  a family $\cF$ of $s$ subsets of $[n]$, how many $k$-chains must $\cF$ contain?  We denote this minimum by $c_k(n,s)$, and determine it for a wide range of values of $s$.  This provides a quantitative strengthening of the Erd\H{o}s result on the size of $k$-chain-free families.

We shall now discuss the background of Sperner's Theorem and this problem further, before presenting our new results.

\subsection{Background}

Extremal set theory is one of the most rapidly developing areas in combinatorics, having applications to other branches of mathematics and computer science including discrete geometry, functional analysis, number theory and complexity.  The typical extremal problem has the following form: how large can a structure be without containing some forbidden configuration?  A classical example, considered by many to be the starting point of extremal set theory, is a theorem of Sperner \cite{sperner}.  An \emph{antichain} is a family of subsets of $[n]$ that does not contain sets $F_1 \subset F_2$.  Sperner's Theorem states that the largest antichain has $\sperner$ sets, a bound that is easily seen to be tight by considering the family of sets of size $\floor{ \frac{n}{2} }$.  This celebrated result enjoys numerous applications and has many extensions, many of which are discussed in Engel's book \cite{En}.  One particular extension, due to Erd\H{o}s \cite{erd45}, shows that the size of the largest set family without a \emph{$k$-chain}, that is, $k$-sets $F_1 \subset F_2 \subset \hdots \subset F_k$, is the sum of the $k-1$ largest binomial coefficients, $M_{k-1} = \sum_{i = \ceil{\frac{n-k+2}{2}}}^{\ceil{\frac{n+k-2}{2}}} \binom{n}{i}$.  When $k=2$, we recover Sperner's Theorem.

Our problem is what we refer to as an Erd\H{o}s-Rademacher-type extension of Erd\H{o}s' theorem, a name we now explain.  Arguably the most well-known result in extremal combinatorics is a theorem of Mantel \cite{mantel} from 1907, which states that an $n$-vertex triangle-free graph can have at most $\floor{ \frac{n^2}{4} }$ edges.  In an unpublished result, Rademacher strengthened this theorem by showing that any graph with $\floor{ \frac{n^2}{4} } + 1$ edges must contain at least $\floor{ \frac{n}{2} }$ triangles.  Erd\H{o}s \cite{erd62a} then extended this to graphs with a linear number of extra edges, and in \cite{erd62b} studied the problem for larger cliques.  More generally, for any extremal problem, the corresponding Erd\H{o}s-Rademacher problem asks how many copies of the forbidden configuration must appear in a structure larger than the extremal bound.

In the context of Sperner's Theorem, this problem was first considered by Erd\H{o}s and Katona, who conjectured that a family with $\sperner + t$ sets must contain at least $t \ceil{
\frac{n+1}{2}}$ $2$-chains.  Kleitman \cite{kleitman} confirmed the conjecture, and, in a far-reaching generalization, showed the minimum number of $2$-chains in a family of any fixed size
is obtained by choosing sets of size as close to $\frac{n}{2}$ as possible.  He then conjectured (see \cite{EK, kleitman}) that the same families minimize the number of
$k$-chains, a problem that has remained open for nearly fifty years.

\begin{CONJ}\label{conj:kleitman}
The number of $k$-chains in a family is minimized by choosing sets of sizes as close to $\frac{n}{2}$ as possible.
\end{CONJ}

\subsection{Our results}

In this paper we study these Erd\H{o}s-Rademacher-type extensions of the theorems of Sperner and Erd\H{o}s.  We began by considering the case of $2$-chains, and determined the minimum number of $2$-chains in a family of any number of sets.  Later, we discovered Kleitman had earlier obtained the same result.  However, through slightly more careful calculations, and by introducing an additional argument, we are able to characterize all extremal families, as given below.

\begin{THM} \label{thm:2chains}
Let $\mathcal{F}$ be a family of subsets of $[n]$, with $\card{\cF} = s \ge \sperner$.  Let $r \in \frac{1}{2} \mathbb{N}$ be the unique half-integer such that $\sum_{i = \frac{n}{2} - r + 1}^{\frac{n}{2} + r - 1} \binom{n}{i} < s \le \sum_{i = \frac{n}{2} - r}^{\frac{n}{2} + r} \binom{n}{i}$.  Then $\mathcal{F}$ minimizes the number of $2$-chains if and only if the following conditions are satisfied:
\begin{enumerate}
	\item For every $F \in \cF$, $\frac{n}{2} - r \le \card{F} \le \frac{n}{2} + r$.
	\item For any $A \subset [n]$ with $\frac{n}{2} - r + 1 \le \card{A} \le \frac{n}{2} + r - 1$, we have $A \in \cF$.
	\item If $s \le \sum_{i = \frac{n}{2} - r}^{\frac{n}{2} + r - 1} \binom{n}{i}$, then $\{ F \in \cF : \card{F} = \frac{n}{2} \pm r \}$ forms an antichain.
	\item If $s \ge \sum_{i = \frac{n}{2} - r}^{\frac{n}{2} + r - 1} \binom{n}{i}$, then $\{ F \notin \cF : \card{F} = \frac{n}{2} \pm r \}$ forms an antichain.
\end{enumerate}
\end{THM}

Our main results verify Conjecture \ref{conj:kleitman} for families of certain sizes.  To begin with, recall that Erd\H{o}s showed the largest family without $k$-chains consists of the $k-1$ middle levels of the hypercube, whose size we denote by $M_{k-1}$.  If we were to add one set to this family, the best we could do would be to add it to the $k$th level, in which case we would create $\binom{\floor{(n+k)/2}}{k-1} (k-1)!$ $k$-chains.  Indeed, we show that every additional set must contribute at least this many new $k$-chains, and the above construction shows this is tight when our extremal family is contained within the $k$ middle levels.

\begin{THM} \label{thm:klevels}
If $\cF$ is a set family over $[n]$ of size $s = M_{k-1} + t$, then $\cF$ contains at least $t \binom{ \floor{ (n + k) / 2 }}{k-1} (k-1)!$ $k$-chains.
\end{THM}

We are then able to extend our argument to work for larger set families, obtaining a result that is tight when the extremal family is contained within the $k+1$ middle levels.

\begin{THM} \label{thm:k+1levels}
Provided $n \ge 15$ and $k \le n - 6$, if $\cF$ is a set family over $[n]$ of size $s = M_k + t$, then the number of $k$-chains in $\cF$ is at least
\[ \binom{n}{\ceil{ (n-k)/2}} \binom{\floor{ (n + k) / 2}}{k-1} (k-1)! + t \left( \binom{\ceil{ (n+k)/2 }}{k-1} + \binom{\ceil{(n+k)/2}}{k} \binom{k}{2} \right)(k-1)!. \]
\end{THM}

In both cases, we actually obtain stronger results (see Theorems \ref{thm:kstability} and \ref{thm:allchains} respectively), providing stability versions of the above theorems, showing that if a family has close to the minimum number of $k$-chains, it must be close in structure to the extremal example.
These stability results are of interest even in the case $k=2$, as one does not obtain any stability from the Kleitman proof for $2$-chains.
We then use the stability results to show that when the above bounds are tight, the extremal families are exactly as in Theorem \ref{thm:2chains}.

\subsection{Outline and notation}

The remainder of this paper is organized as follows.  Section \ref{sec:2chains} contains a proof of Theorem \ref{thm:2chains}.  In Section \ref{sec:klevels}, we prove Theorem \ref{thm:klevels}, and then in Section \ref{sec:k+1levels} prove Theorem \ref{thm:k+1levels}.  In the final section we present some concluding remarks and open problems.  Appendix \ref{app:shifting} contains the proof of a technical proposition needed for Theorem \ref{thm:2chains}.

\medskip

We let $[n]$ denote the set of the first $n$ integers.  For a ground set $X$ and an integer $i$, we denote the family of $i$-subsets of $X$ by $\binom{X}{i} = \{ Y \subseteq X : |Y| = i \}$.  We let $M_k = \sum_{i = \ceil{n-k+1 / 2}}^{\ceil{n+k-1 / 2}} \binom{n}{i}$ be the size of the $k$ middle, and thus largest, levels.  Given a family $\cF$ of subsets of $[n]$, we let $\cF_i = \cF \cap \binom{[n]}{i}$ denote those sets in $\cF$ of size $i$.  The $\ell$-shadow of a family is given by $\partial^{\ell} \cF = \{ G : \exists F \in \cF, G \subset F, |G| = |F| - \ell \}$.  For a subset $F \subset [n]$, we define $m(F) = \max \{ \card{F}, n - \card{F} \}$.

Given a set family $\cF$, $c_k(\cF)$ denotes the number of $k$-chains in $\cF$.  For any $n \in \mathbb{N}$ and $0 \le s \le 2^n$, we let $c_k(n,s)$ denote the minimum of $c_k(\cF)$ over all families $\cF$ of $s$ subsets of $[n]$.  When $k=2$, if we have two families $\cF$ and $\cG$, then we let $c_2(\cF, \cG)$ denote the number of $2$-chains with one set from $\cF$ and one set from $\cG$.

\section{Counting $2$-chains} \label{sec:2chains}

In this section we will prove Theorem \ref{thm:2chains}, characterizing those families that minimize the number of $2$-chains.  We essentially show that it is optimal to take sets of sizes as close to $\frac{n}{2}$ as possible.  The theorem then prescribes how the boundary sets can be distributed.

\medskip

Sperner's Theorem shows that the largest antichain is given by one of the middle levels, that is either all sets of size $\floor{\frac{n}{2}}$ or all sets of size $\ceil{\frac{n}{2}}$.  Obviously, an antichain minimizes the number of $2$-chains, as it has none.  This theorem is then a natural extension of Sperner's Theorem, as it shows that to construct a family of any size that minimizes the number of $2$-chains, one should start by taking sets of size $\floor{\frac{n}{2}}$, then sets of size $\floor{\frac{n}{2}} + 1$, then $\floor{\frac{n}{2}} - 1$, and so on until one has a family of the desired size.  As we shall show, these families are optimal, and so we may denote the number of $2$-chains in the first $s$ such sets by $c_2(n,s)$.

The idea behind the proof is as follows.  If our family $\cF$ contains a set $F$ that is too far away from the middle (i.e. $\card{ \card{F} - \frac{n}{2} } > r$), then we will show that we can shift $F$ closer to the middle and decrease the number of $2$-chains.  Once we have our family contained in the $2r+1$ middle layers, a simple counting argument will give the characterization of extremal families.  As the shifting process is essentially the same as in Kleitman's proof in \cite{kleitman}, we relegate the proof of the following proposition to Appendix \ref{app:shifting}.

\begin{PROP} \label{prop:shifting}
Let $\cF$ be a family of $s > \sperner$ subsets of $[n]$ minimizing the number of $2$-chains.  If $A \in \cF$ is of maximal cardinality, with $|A| = \frac{n}{2} + m$, then for any $B \subset A$, $|B| \ge \frac{n}{2} - m + 1$, we have $B \in \cF$.
\end{PROP}

Assuming this proposition, we shall proceed to prove Theorem \ref{thm:2chains}.

\begin{proof}[Proof of Theorem \ref{thm:2chains}]
We prove the theorem by induction on $s$.

\vspace{0.1in}

For the base case, we take $s = \sperner$.  By Sperner's Theorem, it follows that any family $\cF$ of this size that minimizes the number of $2$-chains must be an antichain.  It is well known that the only antichains of this size are the family of all sets of size $\floor{\frac{n}{2}}$, or the family of sets of size $\ceil{ \frac{n}{2} }$.  It is easy to see that these families are the only ones satisfying Properties $1$ through $4$, with $r = 0$ or $\frac12$ depending on whether $n$ is even or odd respectively.

\vspace{0.1in}

For the induction step, assume $s > \sperner$, and let $\cF$ be an optimal family of size $s$.  Suppose Property 1 were not satisfied.  Since $\cF$ and $\cF' = \{ [n] \setminus F : F \in \cF \}$ have the same number of $2$-chains, we may assume there is a largest set $F \in \cF$ with $|F| = \frac{n}{2} + t$ for some $t > r$.  By Proposition \ref{prop:shifting}, it follows that for every $G \subset F$, $|G| \ge \frac{n}{2} - t + 1$, we have $G \in \cF$.  Hence $F$ is in at least $\sum_{i = 1}^{2t - 1} \binom{\frac{n}{2} + t}{i}$ $2$-chains in $\cF$.  Since $\cF \setminus \{F\}$ is a family of $s-1$ sets, there are at least $c_2(n,s-1)$ $2$-chains in $\cF$ not involving $F$.
Thus the number of $2$-chains in $\cF$ is at least $c_2(n,s-1) + \sum_{i=1}^{2t-1} \binom{\frac{n}{2} + t}{i} > c_2(n,s-1) + \sum_{i=1}^{2r} \binom{\frac{n}{2} + r}{i} \ge
c_2(n,s)$, and so $\cF$ cannot be optimal, giving a contradiction.  Hence if $\cF$ is optimal, each $F \in \cF$ has $\frac{n}{2}
- r \le |F| \le \frac{n}{2} + r$, and so Property 1 is established.

Now consider the case $\sum_{i = \frac{n}{2} - r + 1}^{\frac{n}{2} + r - 1} \binom{n}{i} < s \le \sum_{i = \frac{n}{2} - r}^{\frac{n}{2} + r - 1} \binom{n}{i}$.  Since $s > \sum_{i = \frac{n}{2} - r + 1}^{\frac{n}{2} + r - 1} \binom{n}{i}$, and in light of Property 1, it follows that there exists some $F \in \cF$ with $|F| = \frac{n}{2} \pm r$; by symmetry, we may assume $|F| = \frac{n}{2} + r$.  By Proposition \ref{prop:shifting}, $F$ must be contained in $\sum_{i = 1}^{2r - 1} \binom{\frac{n}{2} + r}{i} = c_2(n,s) - c_2(n,s-1)$ $2$-chains with sets in $\cF$ of sizes between $\frac{n}{2} - r + 1$ and $\frac{n}{2} + r - 1$.  If $F$ is contained in any $2$-chains with sets of size $\frac{n}{2} - r$, then by induction it follows that $\cF$ has more than $c_2(n,s)$ $2$-chains, contradicting the optimality of $\cF$.  Thus $F$ is incomparable to the other sets in $\cF$ of sizes $\frac{n}{2} \pm r$.  Removing $F$, we find that $\cF \setminus \{ F \}$ must also be optimal, and thus Properties 2 and 3 follow.

Finally, suppose $\sum_{i = \frac{n}{2} - r}^{\frac{n}{2} + r - 1} \binom{n}{i} \le s \le \sum_{i = \frac{n}{2} - r}^{\frac{n}{2} + r} \binom{n}{i}$.  By Property 1, we know all sets in $\cF$ have sizes between $\frac{n}{2} - r$ and $\frac{n}{2} + r$.  Let $\cH = \cup_{i = \frac{n}{2} - r}^{\frac{n}{2} + r} \binom{[n]}{i}$ be the family of all subsets of $[n]$ of sizes between $\frac{n}{2} - r$ and $\frac{n}{2} + r$, and let $\cG = \cH \setminus \cF$ be those sets not in $\cF$.  We have $c_2(\cF) = c_2(\cH) - c_2(\cG, \cH) + c_2(\cG)$.  $c_2(\cH)$ depends only on $r$, and hence on $s$, and is independent of the structure of $\cF$.  We have
\[ c_2(\cG, \cH) = \sum_{G \in \cG} c_2( \{ G \}, \cH ) = \sum_{G \in \cG} \left( \sum_{i = \frac{n}{2} - r}^{|G| - 1} \binom{|G|}{i} + \sum_{i = |G| + 1}^{\frac{n}{2} + r} \binom{n - |G|}{i - |G|} \right). \]
The parenthetical term is maximized when $|G| = \frac{n}{2} \pm r$, and so $c_2(\cG,\cH)$ is maximized when for every $G \in \cG$ we have $|G| = \frac{n}{2} \pm r$.  Finally, $c_2(\cG)$ is minimized when $\cG$ is an antichain, in which case $c_2(\cG) = 0$.  Both of these conditions are satisfied by the construction outlined at the beginning of this section, and hence must also be true of any other extremal family.  Thus to minimize $c_2(\cF)$, $\cF$ must contain all sets of sizes between $\frac{n}{2} - r + 1$ and $\frac{n}{2} + r - 1$, and $\cG = \binom{[n]}{\frac{n}{2} + r} \cup \binom{[n]}{\frac{n}{2} - r} \setminus \cF$ should be an antichain, establishing Properties 2 and 4.
This completes the induction step, and with it the proof of Theorem \ref{thm:2chains}.

\end{proof}

\section{Counting $k$-chains} \label{sec:klevels}

We now seek a similar result for $k$-chains, and thus to make some progress on Conjecture \ref{conj:kleitman}.  In this section we verify the conjecture when the number of sets is at most that in the $k$ middle levels, and in the next section we shall extend the result to the $k+1$ middle levels.

\medskip

Note that if we take all $M_{k-1}$ sets in the $k-1$ middle levels, that is of sizes between $\floor{ \frac{n-k}{2}} + 1$ and $\floor{ \frac{n+k}{2}} - 1$, and then add $t$ sets of size $\floor{\frac{n+k}{2}}$, we would create precisely $t \binom{\floor{(n+k)/2}}{k-1} (k-1)!$ $k$-chains.  Hence Theorem \ref{thm:klevels} is tight when $0 \le t \le \binom{n}{\floor{ (n+k)/2 }}$.  We shall in fact prove the following stronger theorem, which provides a stability result.  In the following notation, we let $r \in \left\{ \frac{k-1}{2}, \frac{k}{2} \right\}$ be such that the sets in the $k$ middle levels have sizes between $\frac{n}{2} - r$ and $\frac{n}{2} + r$, and for a set $F$, we define $m(F) = \max \{ \card{F}, n - \card{F} \}$.

\begin{THM} \label{thm:kstability}
Let $\cF$ be a family of subsets of $[n]$.  Then the number of $k$-chains in $\cF$ is bounded by
\[ c_k(\cF) \ge \left( \sum_{ i = n/2 - r}^{n /2 + r} \frac{\card{\cF_i}}{\binom{n}{i}} - (k - 1) \right) \binom{n}{n / 2 + r } \binom{ n / 2 + r }{k - 1} (k-1)! + \sum_{\substack{F \in \cF : \\ \card{ \card{F} - n / 2 } \ge r+1}} \binom{ m(F) }{k-1} (k-1)!. \]
\end{THM}

\begin{proof}[Proof of Theorem \ref{thm:kstability}]
We prove the theorem by induction on $\card{\cF}$.  If $\card{\cF} = 0$, then there is nothing to show, as the desired lower bound is negative.

\vspace{0.1in}

For the induction step, we begin by noting that for every set $F \in \cF$ with $\card{ \card{F} - \frac{n}{2} } \ge r$, $F$ can be in at most $\binom{m(F)}{k-1} (k-1)!$ $k$-chains.  If not, then we could remove $F$, and applying the inductive hypothesis to $\cF \setminus \{F \}$, we would have the desired inequality.

We now use an LYM-type inequality, counting the number of $k$-chains in our family by considering permutations.  We say that a permutation $\sigma \in S_n$ \emph{contains} a set $F \subset [n]$, denoted $F \in \sigma$, if $\{ \sigma(1), \sigma(2), \hdots, \sigma(\card{F}) \} = F$; that is, $F$ is an initial segment of $\sigma$.  Note that if $\sigma$ contains $k$ sets $F_1, F_2, \hdots, F_k$, then those $k$ sets must form a $k$-chain.  For any set $F \subset [n]$, we let $S_n[F] = \{ \sigma \in S_n : F \in \sigma \}.$  Since every permutation containing $m$ sets contributes $m - \binom{m}{k} \le k-1$ to the right-hand side of the sum below, it follows that
\begin{equation} \label{ineq:LYM1}
(k-1) n! \ge \sum_{F \in \cF} \card{S_n[F]} - \sum_{F_1 \subset F_2 \subset \hdots \subset F_k \in \cF} \card{ \cap_{i=1}^k S_n[F_i] }.
\end{equation}

As the second sum is over all $k$-chains in our family $\cF$, this inequality will allow us to bound the number of $k$-chains.  Note that for any $F \in \cF$, we
have $\card{S_n[F]} = \card{F}! \card{[n] \setminus F}!$, and for a $k$-chain $F_1 \subset F_2 \subset \hdots \subset F_k$, $\card{\cap_{i=1}^k S_n[F_i]} =
\card{F_1}! \prod_{i=1}^{k-1} \card{F_{i+1} \setminus F_i }! \card{[n] \setminus F_k}!$ gives the number of permutations containing the $k$-chain.

We shall associate every $k$-chain in $\cF$ with either its minimum or maximum set, depending on which is further away from the middle level.  For $F \in \cF$ with $|F| < \frac{n}{2}$, let
$\cC(F) = \{ F \subset F_2 \subset \hdots \subset F_k : |F| + |F_k| < n \}$, and if $F \in \cF$ with $|F| \ge \frac{n}{2}$, let $\cC(F) = \{ F_1 \subset
\hdots \subset F_{k-1} \subset F : |F_1| + |F| \ge n \}$, and, for convenience, define $C(F) = \card{\cC(F)}$.  Note that we have partitioned the set of
$k$-chains $\{ F_1 \subset F_2 \subset \hdots \subset F_k : F_i \in \cF \}$ into the disjoint sets $\left\{ \cC(F) \right\}_{F \in \cF}$.  We can thus
rewrite inequality (\ref{ineq:LYM1}) as follows:
\begin{equation}
\label{ineq:LYM2}
(k-1) n! \ge \sum_{F \in \cF} \left[ \card{S_n[F]} - \sum_{F_1 \subsetneq F_2 \subsetneq \hdots
\subsetneq F_k \in \cC(F)} \card{ \cap_{i = 1}^k S_n[F_i] } \right].
\end{equation}

To bound $\card{ \cap_{i=1}^k S_n[F_i] }$ appropriately, we require that $\emptyset$ and $[n]$ not be members of our family.  This is given by the following lemma.

\begin{LEMMA} \label{lem:noverylargesets}
If $\card{\cF} \le M_{n-1}$, and $\cF$ minimizes the number of $k$-chains, then $\emptyset \notin \cF$ and $[n] \notin \cF$.
\end{LEMMA}

\begin{proof}
Suppose we had $[n] \in \cF$.  Since $\card{F} \le 2^n - 2$, there must be some $\emptyset \neq F \notin \cF$.  We decrease the number of $k$-chains in $\cF$ by replacing $[n]$ with $F$, since any new $k$-chain involving $F$ was a $k$-chain with $[n]$ before.

Similarly, if $\emptyset \in \cF$, we can replace it with any set $[n] \neq F \notin \cF$.
\end{proof}

Note that $M_{n-1} = 2^n - 2$.  Thus, if $\card{\cF} > M_{n-1}$, then either we have all subsets of $[n]$, or our family is missing just one set, in which case
(as we explained above) it is best to remove either $\emptyset$ or $[n]$.  In either case, the bound in Theorem \ref{thm:kstability} remains true.

\medskip

We now assume $1 \le \card{F} \le n-1$ for all $F \in \cF$.  If we fix $F_1, |F_1|< n/2$, then we maximize $\card{ \cap_{i=1}^k S_n[F_i]} = \card{F_1}! \prod_{i=1}^{k-1} \card{F_{i+1} \setminus F_i}! \card{[n] \setminus F_k}!$ by taking
$\card{F_{i+1} \setminus F_{i}} = 1$ for $1 \le i \le k-1$, and so $\card{\cap_{i=1}^k S_n[F_i] } \le \card{F_1}! \left( n - \card{F_1} - (k-1) \right)!$.  The same holds true if we instead fix $F_k, |F_k| \geq n/2$, and thus $\card{ \cap_{i=1}^k S_n[F_i]} \le \left( \card{F_k} - (k-1) \right)! \left(n - \card{F_k} \right)!$.  We can unify both bounds in the form $\frac{\card{F}! (n-\card{F})!}{\binom{m(F)}{k-1}(k-1)!}$.

Moreover, by definition we must have $\cC(F) = \emptyset$ for any $F$ with $\card{ \card{F} - \frac{n}{2}} \le r-1$.  Hence we split our sum based on how $\card{ \card{F} - \frac{n}{2} }$ compares to $r$. Dividing through by $n!$, inequality \eqref{ineq:LYM2} leads to
\[ k-1 \ge \sum_{F \in \cF} \frac{1}{\binom{n}{\card{F}}} \left( 1 - \frac{C(F)}{\binom{m(F)}{k-1} (k-1)!} \right) = \Sigma_1 + \Sigma_2 + \Sigma_3, \]
where
\[ \Sigma_1 = \sum_{\substack{ F \in \cF \\ \card{ \card{F} - n / 2 } \le r - 1}}  \frac{1}{\binom{n}{\card{F}}} = \sum_{i = n/2 - r + 1}^{n/2 + r - 1} \frac{ \card{\cF_i}}{\binom{n}{i}}, \]
\[ \Sigma_2 = \sum_{\substack{ F \in \cF \\ \card{ \card{F} - n/2 } = r}}  \frac{1}{\binom{n}{n/2 + r}} \left( 1 - \frac{C(F)}{\binom{n/2 + r}{k-1} (k-1)!} \right) = \frac{\card{\cF_{n/2 - r}} + \card{\cF_{n/2 + r}}}{\binom{n}{n/2 + r}} - \frac{ \sum_{\substack{F \in \cF \\ \card{ \card{F} - n/2} = r}} C(F) }{\binom{n}{n/2 + r} \binom{n/2 + r}{k-1} (k-1)!}, \]
and
\[ \Sigma_3 = \sum_{\substack{ F \in \cF \\ \card{\card{F} - n / 2} \ge r + 1}} \frac{1}{\binom{n}{\card{F}}} \left( 1 - \frac{C(F)}{\binom{m(F)}{k-1}} \right). \]

Note that $\sum_{\card{\card{F} - n/2}= r} C(F) = c_k(\cF) - \sum_{\card{ \card{F} - n/2} \ge r+1} C(F)$, and so, substituting in $\Sigma_2$, we obtain
\begin{align*}
	\frac{c_k(\cF)}{\binom{n}{n/2 + r} \binom{n/2+r}{k-1} (k-1)!} + k-1 &\ge \sum_{i = n/2-r}^{n/2 + r} \frac{\card{\cF_i}}{\binom{n}{i}} \\
	&\quad + \sum_{\substack{F \in \cF : \\ \card{ \card{F} - n/2 } \ge r + 1}} \left( \frac{1}{\binom{n}{\card{F}}} - \frac{C(F)}{(k-1)!} \left( \frac{1}{\binom{n}{\card{F}} \binom{m(F)}{k-1}} - \frac{1}{\binom{n}{n/2 + r} \binom{n/2 + r}{k-1}} \right) \right).
\end{align*}

Now, since $\binom{n}{a} \binom{a}{k-1} = \binom{n}{k-1} \binom{n-k+1}{a - k + 1}$, it follows that when $m(F) \ge \frac{n}{2} + r$, $\binom{n}{\card{F}} \binom{m(F)}{k-1} \le \binom{n}{n/2 + r} \binom{n/2 + r}{k-1}$, as $\frac{1}{2} (n - k + 1) \le \frac{n}{2} + r - k + 1 \le m(F) - k + 1$.  Hence the summand is minimized when $C(F)$ is as large as possible, which, by our inductive hypothesis, is $\binom{m(F)}{k-1} (k-1)!$.  Substituting this into the inequality above gives
\[ \frac{c_k(\cF)}{\binom{n}{n/2 + r} \binom{n/2 + r}{k-1} (k-1)!} + k - 1 \ge \sum_{i=n/2 - r}^{n/2 + r} \frac{\card{\cF_i}}{\binom{n}{i}} + \sum_{\substack{F \in \cF : \\ \card{ \card{F} - n/2 } \ge r+1}} \frac{\binom{m(F)}{k-1}}{\binom{n}{n/2 + r} \binom{n/2 + r}{k-1}}. \]

Rearranging gives the desired bound
\[ c_k(\cF) \ge \left( \sum_{i = n/2 - r}^{n/2 + r} \frac{ \card{\cF_i}}{\binom{n}{i}} - (k-1) \right) \binom{n}{n/2 + r} \binom{n/2 + r}{k-1} (k-1)! + \sum_{\substack{F \in \cF : \\ \card{\card{F} - n / 2 } \ge r + 1}} \binom{m(F)}{k-1} (k-1)!. \]
\end{proof}

Given Theorem \ref{thm:kstability}, it is easy to deduce Theorem \ref{thm:klevels}.  In fact, we are able to characterize all extremal families.

\begin{proof}[Proof of Theorem \ref{thm:klevels}]
Suppose we have a family of sets $\cF$, with $\card{\cF} = M_{k-1} + t$.  Note that the contribution each set $F \in \cF$ makes to the right-hand side above is $\frac{\binom{n}{n/2 + r}}{ \binom{n}{\card{F}}} \binom{n/2 + r}{k-1} (k-1)!$ if $\frac{n}{2} - r \le \card{F} \le \frac{n}{2} + r$, and $\binom{m(F)}{k-1} (k-1)!$ otherwise.  This contribution increases with $\card{ \card{F} - \frac{n}{2}}$, and so to minimize the right-hand size we need all sets to satisfy
$\card{ \card{F} - \frac{n}{2}} \leq r$. Moreover, since the binomial coefficients $\binom{n}{i}$ are minimized over $n/2 - r \le i \le n/2 + r$ when
$i=n/2 \pm r$, any extremal family must contain all sets of sizes between $\frac{n}{2} - r + 1$ and $\frac{n}{2} + r - 1$, with the remaining sets having size $\frac{n}{2} \pm r$.  It is easy to see that such a collection of sets gives $c_k(\cF) \ge t \binom{n/2 + r}{k-1} (k-1)!$ above, as required.

To classify the extremal families, note that we already know we must have all sets of sizes between $\frac{n}{2} - r + 1$ and $\frac{n}{2} + r - 1$.  If $r = \frac{k}{2}$, this consists of the middle $k-1$ levels, giving $M_{k-1}$ sets.  Hence we have $t$ sets of size $\frac{n}{2} \pm r$.  To obtain equality in \eqref{ineq:LYM1}, we must have every chain passing through either $k-1$ or $k$ sets of $\cF$.  As all the sets in the $k-1$ middle levels are in $\cF$, the chain can contain at most one set from $\cF$ of size $\frac{n}{2} \pm r$.  From this, we deduce that $\{ F \in \cF : \card{F} = \frac{n}{2} \pm r\}$ must be an antichain.

If $r = \frac{k-1}{2}$, we know the middle $k-2$ levels are full, and we have $\binom{n}{n/2 + r} + t$ sets of size $\frac{n}{2} \pm r$.  In order to obtain equality in \eqref{ineq:LYM1}, we must therefore have every chain pass through at least one set in $\cF$ of size $\frac{n}{2} \pm r$.  Hence $\{ G \notin \cF : \card{G} = \frac{n}{2} \pm r \}$ must form an antichain.

In particular, we note that the extremal families are exactly the same as for Theorem \ref{thm:2chains}.
\end{proof}

We remark that Theorem \ref{thm:kstability} is a stability result for Theorem \ref{thm:klevels}, as our bound on $c_k(\cF)$ increases if we are missing sets with $\card{ \card{F} - \frac{n}{2} } \le r-1$, or have sets with $\card{ \card{F} - \frac{n}{2} } \ge r+1$.  Moreover, the stability estimates we obtain can also be tight.  For instance, if we replace $\ell$ sets of size $\frac{n}{2} \pm r$ with sets of size $\frac{n}{2} + r + 1$, Theorem \ref{thm:kstability}, together with some simple computations, shows that we should gain at least $\ell \binom{n/2 + r}{k-2} (k-1)!$ extra $k$-chains.
Moreover, it is easy to check that we gain precisely that many $k$-chains in the case when our family includes all of the $\left( \frac{n}{2} + r \right)$-sets and none of the $\left( \frac{n}{2} - r \right)$-sets in the shadow of the $\left( \frac{n}{2} + r + 1 \right)$-sets which were added.

Similarly, if we replace $\ell$ sets of size $\frac{n}{2} + r - 1$ with sets of size $\frac{n}{2} + r$, the theorem shows that we must gain at least $\frac{2r - 1}{n/2 + r} \ell \binom{n/2 + r}{k-1} (k-1)!$ extra $k$-chains. This is tight again, if our family includes all the sets of size $\frac{n}{2} + r$ containing any of the replaced sets.  As we remarked earlier, these stability results are new even in the case $k=2$.

\section{Larger families} \label{sec:k+1levels}

While Theorem \ref{thm:klevels} provides a tight bound on the number of $k$-chains appearing in families contained within the $k$ middle levels, it underestimates the number of $k$-chains appearing in larger families.  This is because in our calculations we assumed every $k$-chain had steps ($F_i \setminus F_{i-1}$) of size $1$, as this maximizes the number of permutations containing the $k$-chain.  However, when we are working with the $k+1$ middle levels, we also have $k$-chains with a larger step of size $2$, and so we shall have to make our argument more robust in order to handle these chains.

However, there is one additional difficulty.  Recall that the number of permutations containing a $k$-chain $F_1 \subset F_2 \subset \hdots \subset F_k$ is given by $\card{\cap_{i=1}^k S_n[F_i]} = \card{F_1}! \prod_{i=1}^{k-1} \card{ F_{i+1} \setminus F_i }! \left( n - \card{F_k} \right)!$.  If we fix $F_k$, say, then we would hope that for $k$-chains involving a larger step, the largest this can be is to have one step of size $2$, and have all the other steps have size $1$, as this is precisely the type of $k$-chain that appears in our extremal families.  Such $k$-chains are in $2 ( \card{F_k} - k )! (n - \card{F_k})!$ permutations.

Unfortunately, a chain with $k-2$ steps of size $1$ and one step of size $\card{F_k} - k + 1$ (so that $\card{F_1} = 1$) is contained in $( \card{F_k} - k + 1)! (n - \card{F_k})!$ permutations, which is larger than the bound we require.  However, recall that we assign $k$-chains to either $F_1$ or $F_k$, depending on which is further from the middle.  Thus, unless $\card{F_k} = n-1$, we would assign the above $k$-chain to $F_1$, and so we would be fixing $F_1$ and not $F_k$.  Hence the one case we need to avoid is having a $k$-chain starting with a set of size $1$ and ending with a set of size $n-1$.

We shall later provide a separate argument to show that there cannot be any sets of size $n-1$ in an extremal family, thus bypassing this problem.  In the meanwhile, for the purposes of our stability results, we shall assume there are no sets of size $n-1$.  We now require two arguments - one to bound the number of $k$-chains with larger steps, and one to bound the total number of $k$-chains.

We introduce the following notation for the remainder of this section.  Given $n$ and $k$, we let $a = \ceil{ \frac{n+k}{2} }$, so that the $k$ middle levels are those sets of sizes between $a-k$ and $a-1$, and the $(k+1)$st middle level has sets of size $a$.  For a set family $\cF$, we let $\cC(\cF)$ be the set of $k$-chains in $\cF$.  We partition these into two subsets: $\cC_1(\cF)$ are those $k$-chains $F_1 \subset F_2 \subset \hdots \subset F_k$ with $\card{F_{i+1} \setminus F_i} = 1$ for all $1 \le i \le k-1$, and $\cC_2(\cF) = \cC(\cF) \setminus \cC_1(\cF)$ those $k$-chains with a larger step.  We let $C(\cF)$, $C_1(\cF)$ and $C_2(\cF)$ denote the number of $k$-chains in these subsets respectively.  As in Theorem \ref{thm:klevels}, we will again identify a $k$-chain with one of its endpoints $F_1$ or $F_k$, depending which is further from the middle level, giving the partition $\{ \cC(F) \}_{F \in \cF}$ of $\cC(\cF)$.  These sets will again be partitioned into $\cC_1(F)$ and $\cC_2(F)$, depending on whether or not the $k$-chains have a step of size at least $2$.  Finally, $C(F)$, $C_1(F)$ and $C_2(F)$ represent the sizes of the corresponding sets of $k$-chains.

\subsection{Counting $k$-chains with larger steps} \label{subsec:largesteps}

We begin by showing that large families must contain a number of $k$-chains with a step of size at least $2$.  The following proposition also provides some stability, which we shall require to show that an extremal family cannot contain any sets of size $n-1$.

\begin{PROP} \label{prop:stepchains}
Let $\cF$ be a set family of size $\card{\cF} = M_k + t_1$, with $1 \le \card{F} \le n-2$ for all $F \in \cF$, and with at least $t_2$ sets missing from the middle $k-1$ levels.  Then
\[ C_2(\cF) \ge \left( t_1 + \left( \frac{k-1}{a} \right) t_2 \right) \binom{a}{k} \binom{k}{2} (k-1)!. \]
\end{PROP}

\begin{proof}
We prove the statement by induction on $t_1 + t_2 \ge 0$, noting that we must have $t_2 \ge 0$.  The base case of $t_1 + t_2 = 0$ is trivial, as in this case the right-hand side is non-positive.

\medskip

The proof will now run along very similar lines to that of Theorem \ref{thm:klevels}, and we shall just make a few changes to count only those chains with a large step.  To begin with, when we are counting sets and $k$-chains in permutations, we only want to consider those $k$-chains with a large step.  To ensure this, we shall not count $k$-chains that appear consecutively in some permutation.  That is, if $\sigma \in S_n$ contains the sets $F_1 \subset F_2 \subset \hdots \subset F_s$ for some $s \ge k+1$, we will count $F_1 \subset F_2 \subset \hdots \subset F_{k-1} \subset F_{k+1}$, but not $F_1 \subset F_2 \subset \hdots \subset F_k$.  Thus every $k$-chain we consider is bound to have some step of size at least $2$.  If $s \ge k$, then the number of such chains is $\binom{s}{k} - \left( s - (k-1) \right)$, and since $(k-1)s - \left( \binom{s}{k} - \left( s - (k-1) \right) \right) \le k^2 - k$, it follows that
\begin{align*}
	(k^2 - k) n! &\ge (k-1) \sum_{F \in \cF} \card{S_n[F]} - \sum_{F_1 \subset \hdots \subset F_k \in \cC_2(\cF)} \card{ \cap_{i=1}^k S_n[F_i]} \\
	&\ge \sum_{F \in \cF} \left( (k-1) \card{S_n[F]} - \sum_{F_1 \subset \hdots \subset F_k \in \cC_2(F)} \card{ \cap_{i=1}^k S_n[F_i]} \right).
\end{align*}

As before, we now seek to maximize the terms $\card{\cap_{i=1}^k S_n[F_i]}$.  Provided we have $1 \le \card{F} \le n-2$ for all sets $F \in \cF$, if we fix one of the endpoints of the chain, the number of permutations it is contained in is maximized when we have one step of size $2$, and all the other steps of size $1$.
Thus we can bound $\card{\cap_{i=1}^k S_n[F_i]}$ by $2|F_1|!(n-|F_1|-k)!$ or $2(|F_k|-k)!(n-|F_k|)!$. Dividing through by $n!$, we have that
\[ k^2 - k \ge \sum_{F \in \cF} \frac{1}{\binom{n}{\card{F}}} \left( k-1 - \frac{C_2(F)}{\binom{m(F)}{k} \frac{k!}{2}} \right). \]

Now, by definition, we must have $C_2(F) = 0$ for all sets $F$ in the $k$ middle levels; that is, with $a-k \le \card{F} \le a-1$.  Let $\hat{\cF} = \{ F \in \cF : \card{F} \le a - k - 1 \textrm{ or } \card{F} \ge a \}$ be those sets outside the middle $k$ levels.  Thus
\[ k^2 - k \ge \sum_{i = a-k}^{a-1} \frac{(k-1) \card{\cF_i}}{\binom{n}{i}} + \sum_{F \in \hat{\cF}} \frac{1}{\binom{n}{\card{F}}} \left( k-1 - \frac{C_2(F)}{\binom{m(F)}{k} \frac{k!}{2} } \right). \]

We may assume that for every set $F \in \hat{\cF}$, $C_2(F) \le \binom{a}{k} \binom{k}{2} (k-1)! = (k-1) \binom{a}{k} \frac{k!}{2}$, since otherwise we may remove $F$ from $\cF$ and are then done by induction.  Hence, since $m(F) \ge a$ for all $F \in \cF$, the parenthetical term in the second sum is always non-negative, and so the right-hand side is minimized by replacing $\binom{n}{\card{F}}$ by $\binom{n}{a}$, giving
\begin{align*}
	k^2 - k &\ge  \sum_{i = a-k}^{a-1} \frac{(k-1) \card{\cF_i}}{\binom{n}{i}} + \sum_{F \in \hat{\cF}} \frac{1}{\binom{n}{a}} \left( k-1 - \frac{C_2(F)}{\binom{m(F)}{k} \frac{k!}{2} } \right) \\
	&\ge \sum_{i = a-k}^{a-1} \frac{(k-1) \card{\cF_i}}{\binom{n}{i}} + \sum_{F \in \hat{\cF}} \frac{1}{\binom{n}{a}} \left( k-1 - \frac{C_2(F)}{\binom{a}{k} \frac{k!}{2} } \right) \\
	&= \sum_{i = a-k}^{a-1} \frac{(k-1) \card{\cF_i}}{\binom{n}{i}} + \frac{(k-1) | \hat{\cF} | }{\binom{n}{a}} - \frac{C_2(\cF)}{\binom{n}{a} \binom{a}{k} \frac{k!}{2}},
\end{align*}
and so
\[ \frac{C_2(\cF)}{\binom{n}{a} \binom{a}{k} \frac{k!}{2} } \ge \sum_{i = a-k}^{a-1} \frac{(k-1) \card{\cF_i}}{\binom{n}{i}} - (k^2 - k) + \frac{(k-1) | \hat{\cF} | }{\binom{n}{a}}. \]

If the middle $k$ levels were full, then the first sum would equal $k^2 - k$.  Since we must have at least $t_2$ sets missing from the middle $k-1$ levels, the right-hand side is minimized when there are exactly $t_2$ sets missing, all of size $a-1$.  In this case, $| \hat{\cF} | = t_1 + t_2$, giving
\[ \frac{C_2(\cF)}{\binom{n}{a} \binom{a}{k} \frac{k!}{2} } \ge \frac{(k-1)t_1}{\binom{n}{a}} + (k-1)t_2 \left( \frac{1}{\binom{n}{a}} - \frac{1}{\binom{n}{a-1}} \right) = \left( t_1 + \left( \frac{2a-n-1}{a} \right) t_2 \right) \frac{k-1}{\binom{n}{a}}. \]

As $a = \ceil{\frac{n+k}{2}}$, we have $2a - n - 1 \ge k - 1$, and so multiplying through by $\binom{n}{a} \binom{a}{k} \frac{k!}{2}$ gives the desired bound.
\end{proof}

\subsection{Counting all $k$-chains}

As Proposition \ref{prop:stepchains} offers us some control over the number of chains with large steps, we can now proceed to bound the total number of $k$-chains in $\cF$.  Again, our result provides somes stability, as we shall require to forbid sets of size $n-1$.

\begin{THM} \label{thm:allchains}
Let $\cF$ be a set family of size $\card{\cF} = M_k + t_1$, with $1 \le \card{F} \le n-2$ for all $F \in \cF$, and with at least $t_2$ sets missing from the middle $k-1$ levels.  Then
\[ C(\cF) \ge \binom{n}{a-k} \binom{n-a+k}{k-1} (k-1)! + \left( t_1 + \left( \frac{k-1}{a} \right) t_2 \right) \left( \binom{a}{k-1} + \binom{a}{k} \binom{k}{2} \right) (k-1)!. \]
\end{THM}

\begin{proof}
We prove the theorem by induction on $t_1 + t_2 \ge 0$, and again must have $t_2 \ge 0$.  The base case of $t_1 + t_2 = 0$ follows from Theorem \ref{thm:klevels}, as when $\card{\cF} = M_k = M_{k-1} + \binom{n}{a-k}$, the right-hand side above is less than the lower bound for $t = \binom{n}{a-k}$ in Theorem \ref{thm:klevels}.

\medskip

We may now assume that any set $F$ not in the middle $k-1$ levels is contained in at most $\left( \binom{a}{k-1} + \binom{a}{k} \binom{k}{2} \right) (k-1)!$ $k$-chains.  If not, then we may remove $F$ from $\cF$, thus decreasing $t_1$ by $1$.  Applying the inductive hypothesis to $\cF \setminus \{F\}$ and adding the $k$-chains involving $F$ then gives the requisite number of $k$-chains.

We once again seek to bound the number of $k$-chains in our family by counting sets and $k$-chains in permutations, except this time we shall consider all $k$-chains appearing in the permutations.  A permutation with $s$ sets gives rise to $\binom{s}{k}$ $k$-chains, and since for all $s$ we have $ks - \binom{s}{k} \le k^2 - 1$, it follows that
\begin{align} \label{ineq:LYM3}
	(k^2 - 1) n! &\ge k \sum_{F \in \cF} \card{ S_n[F] } - \sum_{F_1 \subset \hdots \subset F_k \in \cC(\cF)} \card{ \cap_{i=1}^k S_n[F_i] } \\
	&= \sum_{F \in \cF} \left( k \card{S_n[F]} - \sum_{F_1 \subset \hdots \subset F_k \in \cC(F)} \card{ \cap_{i=1}^k S_n[F_i] } \right). \notag
\end{align}

To maximize $\card{ \cap_{i=1}^k S_n[F_i] }$, since we have no sets of size $n-1$, we should again take all the gaps to be as small as possible.  Those chains in $\cC_1(\cF)$ all have steps of size $1$, while those in $\cC_2(\cF)$ should have one step of size $2$, and the rest of size $1$.  Dividing by $n!$ gives
\[ k^2 - 1 \ge \sum_{F \in \cF} \frac{1}{\binom{n}{\card{F}}} \left( k - \frac{C_1(F)}{\binom{m(F)}{k-1} (k-1)!} - \frac{C_2(F)}{\binom{m(F)}{k} \frac{k!}{2}} \right). \]

By definition, if $a - k + 1 \le \card{F} \le a - 1$, we must have $C_1(F) = C_2(F) = 0$, and if $\card{F} = a-k$, then $C_2(F) = 0$.  Thus we have three types of $k$-chains to consider: those in $\cC_1(F)$ for $\card{F} = a-k$, those in $\cC_1(F)$ for $F \in \hat{\cF} = \{ F \in \cF : \card{F} \le a - k - 1 \textrm{ or } \card{F} \ge a \}$, and those in $\cC_2(F)$ for $F \in \hat{\cF}$.  Splitting our sums thus, we obtain
\begin{align*}
	k^2 - 1 \ge &\sum_{i=a-k}^{a-1} \frac{k\card{\cF_i}}{\binom{n}{i}} - \frac{ \sum_{\card{F} = a-k} C_1(F) }{\binom{n}{a-k} \binom{n-a+k}{k-1} (k-1)!} + \sum_{F \in \hat{\cF}} \frac{1}{\binom{n}{\card{F}}} \left( 1 - \frac{C_1(F)}{\binom{m(F)}{k-1} (k-1)!} \right) \\
	&+  \sum_{F \in \hat{\cF}} \frac{1}{\binom{n}{\card{F}}} \left( k - 1 - \frac{C_2(F)}{\binom{m(F)}{k} \frac{k!}{2}} \right).
\end{align*}

Since
\[ \sum_{\card{F} = a-k} C_1(F) = C(\cF) - \sum_{F \in \hat{\cF}} C_1(F) - \sum_{F \in \hat{\cF}} C_2(F), \]
we can substitute this expression into the second sum, and redistribute, to obtain
\[ \frac{C(\cF)}{\binom{n}{a-k} \binom{n-a+k}{k-1} (k-1)!} \ge \sum_{i = a-k}^{a-1} \frac{k \card{\cF_i}}{\binom{n}{i}} - (k^2 - 1) + \Sigma_1 + \Sigma_2, \]
where
\begin{align*}
	\Sigma_1 &= \sum_{F \in \hat{\cF}} \frac{1}{\binom{n}{\card{F}}} - \frac{C_1(F)}{\binom{n}{\card{F}} \binom{m(F)}{k-1} (k-1)!} + \frac{C_1(F)}{\binom{n}{a-k} \binom{n-a+k}{k-1} (k-1)!}, \textrm{ and } \\
	\Sigma_2 &= \sum_{F \in \hat{\cF}} \frac{1}{\binom{n}{\card{F}}} \left( k - 1 - \frac{C_2(F)}{\binom{m(F)}{k} \frac{k!}{2} } \right) + \frac{C_2(F)}{\binom{n}{a-k} \binom{n-a+k}{k-1} (k-1)! }.
\end{align*}

The following lemmas, whose proofs we defer to the end of this subsection, allow us to bound these sums.

\begin{LEMMA} \label{lem:c1bound}
For every $F \in \hat{\cF}$, we have
\[ \frac{1}{\binom{n}{\card{F}}} - \frac{C_1(F)}{\binom{n}{\card{F}} \binom{m(F)}{k-1} (k-1)!} + \frac{C_1(F)}{\binom{n}{a-k} \binom{n-a+k}{k-1} (k-1)!} \ge \frac{\binom{a}{k-1}}{\binom{n}{a-k} \binom{n-a+k}{k-1}}. \]
\end{LEMMA}

\begin{LEMMA} \label{lem:c2bound}
For every $F \in \hat{\cF}$, we have
\[ \frac{1}{\binom{n}{\card{F}}} \left( k-1 - \frac{C_2(F)}{\binom{m(F)}{k} \frac{k!}{2}} \right) \ge \frac{1}{\binom{n}{a}} \left( k - 1 - \frac{C_2(F)}{\binom{a}{k} \frac{k!}{2}} \right). \]
\end{LEMMA}

We now replace our summands with these lower bounds, obtaining
\begin{align*}
	\frac{C(\cF)}{\binom{n}{a-k} \binom{n-a+k}{k-1} (k-1)!} &\ge \sum_{i = a-k}^{a-1} \frac{k \card{\cF_i}}{\binom{n}{i}} - (k^2 - 1) + \sum_{F \in \hat{\cF}} \frac{\binom{a}{k-1}}{\binom{n}{a-k} \binom{n-a+k}{k-1}} \\
	&\quad + \sum_{F \in \hat{\cF}} \left( \frac{k-1}{\binom{n}{a}} + C_2(F) \left( \frac{1}{\binom{n}{a-k} \binom{n-a+k}{k-1} (k-1)!} - \frac{1}{\binom{n}{a} \binom{a}{k} \frac{k!}{2}} \right) \right) \\
	&= \sum_{i=a-k}^{a-1} \frac{k \card{\cF_i}}{\binom{n}{i}} - (k^2 - 1) + | \hat{\cF} | \left( \frac{ \binom{a}{k-1} }{ \binom{n}{a-k} \binom{n-a+k}{k-1} } + \frac{k-1}{\binom{n}{a}} \right) \\
	&\quad + \frac{C_2(\cF)}{\binom{n}{a-k} \binom{n-a+k}{k-1} (k-1)!} \left( 1 - \frac{2}{n-a+1} \right).
\end{align*}

We can use Proposition \ref{prop:stepchains} to lower bound $C_2(\cF)$.  Moreover, as $a \ge n-a+k$, it follows that each set in $\hat{\cF}$, whose size is not $a$, has greater weight than any set in the middle $k$ levels.  Thus, the right-hand side is minimized when we fill the middle $k-1$ levels as much as possible.  If we were to have the full $k$ middle levels, the first sum would be equal to $k^2$.  However, as we must have at least $t_2$ sets missing from the middle $k-1$ levels, it is best to have exactly $t_2$ sets of size $a-1$ missing, resulting in
\begin{align*}
	\frac{C(\cF)}{\binom{n}{a-k} \binom{n-a+k}{k-1} (k-1)!} &\ge 1 - \frac{k t_2}{\binom{n}{a-1}} + \left(t_1 + t_2 \right) \left( \frac{\binom{a}{k-1}}{\binom{n}{a-k} \binom{n-a+k}{k-1}} + \frac{k-1}{\binom{n}{a}} \right) \\
	&\quad + \frac{\left(t_1 + \left( \frac{k-1}{a} \right) t_2\right) \binom{a}{k} \binom{k}{2} (k-1)!}{\binom{n}{a-k} \binom{n-a+k}{k-1} (k-1)!} \left(1 - \frac{2}{n-a+1} \right) \\
	&= 1 + A_1 t_1 + A_2 t_2,
\end{align*}
where, after simplifying the binomial expressions, we find
\begin{align*}
	A_1 &= \frac{\binom{a}{k-1}}{\binom{n}{a-k} \binom{n-a+k}{k-1}} + \frac{\binom{a}{k} \binom{k}{2}}{\binom{n}{a-k} \binom{n-a+k}{k-1}} + \frac{k-1}{\binom{n}{a}} - \frac{2 \binom{a}{k} \binom{k}{2}}{(n-a+1) \binom{n}{a-k} \binom{n-a+k}{k-1}} \\
	&= \frac{\binom{a}{k-1} + \binom{a}{k} \binom{k}{2}}{\binom{n}{a-k} \binom{n-a+k}{k-1}}, \textrm{ and } \\
	A_2 &= \frac{\binom{a}{k-1}}{\binom{n}{a-k} \binom{n-a+k}{k-1}} + \frac{(k-1) \binom{a}{k} \binom{k}{2}}{a \binom{n}{a-k} \binom{n-a+k}{k-1}} + \frac{k-1}{\binom{n}{a}} - \frac{k}{\binom{n}{a-1}} - \frac{2(k-1) \binom{a}{k} \binom{k}{2}}{a(n-a+1) \binom{n}{a-k} \binom{n-a+k}{k-1}} \\
	&= \frac{k-1}{a} \left( \frac{\binom{a}{k-1} + \binom{a}{k} \binom{k}{2}}{\binom{n}{a-k}\binom{n-a+k}{k-1}} + \frac{2a-n-k}{\binom{n}{a}} \right) \ge \frac{k-1}{a} \frac{\binom{a}{k-1} + \binom{a}{k} \binom{k}{2}}{\binom{n}{a-k} \binom{n-a+k}{k-1}}.
\end{align*}

Multiplying through by $\binom{n}{a-k} \binom{n-a+k}{k-1} (k-1)!$ gives the desired bound.
\end{proof}

To complete the proof, we now prove the two lemmas.

\begin{proof}[Proof of Lemma \ref{lem:c1bound}]
Note that $\binom{n}{\card{F}} \binom{m(F)}{k-1} = \binom{n}{k-1} \binom{n-k+1}{m(F) - k + 1}$, and, by the same token, $\binom{n}{a-k} \binom{n-a+k}{k-1} = \binom{n}{k-1} \binom{n-k+1}{a - k}$.  Since $a - k = \floor{\frac{n-k+1}{2}}$, and, as $F \in \hat{\cF}$, we have $m(F) \ge a$, it follows that $\binom{n}{\card{F}} \binom{m(F)}{k-1} \le \binom{n}{a-k} \binom{n-a+k}{k-1}$.
Thus the left-hand side of the inequality is minimized when we choose $C_1(F)$ as large as possible.  By definition, $C_1(F) \le \binom{m(F)}{k-1} (k-1)!$.  Making this substitution gives
\[ \frac{1}{\binom{n}{\card{F}}} - \frac{C_1(F)}{\binom{n}{\card{F}} \binom{m(F)}{k-1} (k-1)!} + \frac{C_1(F)}{\binom{n}{a-k} \binom{n-a+k}{k-1} (k-1)!} \ge \frac{\binom{m(F)}{k-1}}{\binom{n}{a-k} \binom{n-a+k}{k-1}} \ge \frac{\binom{a}{k-1}}{\binom{n}{a-k} \binom{n-a+k}{k-1}}. \]
\end{proof}

\begin{proof}[Proof of Lemma \ref{lem:c2bound}]
If $m(F) = a$, then we have equality, so we may assume $m(F) \ge a+1$.  By induction, we can assume that no set is in more than $\left( \binom{a}{k-1} + \binom{a}{k} \binom{k}{2} \right) (k-1)!$ $k$-chains, giving a bound on $C_2(F)$.  Thus
\[ C_2(F) \le \left( \binom{a}{k-1} + \binom{a}{k} \binom{k}{2} \right) (k-1)! \le \binom{a+1}{k} \binom{k}{2} (k-1)! \le (k-1) \binom{m(F)}{k} \frac{k!}{2}, \]
and so the factor $\left( k - 1 - \frac{C_2(F)}{\binom{m(F)}{k} \frac{k!}{2}} \right)$ is non-negative.  As $\binom{n}{a} \ge \binom{n}{\card{F}}$, we thus have
\[ \frac{1}{\binom{n}{\card{F}}} \left( k-1 - \frac{C_2(F)}{\binom{m(F)}{k} \frac{k!}{2}} \right) \ge \frac{1}{\binom{n}{a}} \left( k - 1 - \frac{C_2(F)}{\binom{m(F)}{k} \frac{k!}{2}} \right) \ge \frac{1}{\binom{n}{a}} \left( k - 1 - \frac{C_2(F)}{\binom{a}{k} \frac{k!}{2}} \right). \]
\end{proof}

\subsection{Forbidding large sets}

Given the previous theorem, all that remains is to show that an extremal family cannot contain sets of size $n-1$.  The idea behind this is as follows.  A set of size $n-1$ has a very large shadow in the $k-1$ middle levels.  In order for this set to not contain too many $k$-chains, we must therefore be missing a lot of sets in the $k-1$ middle levels.  By Theorem \ref{thm:allchains}, it then follows that the remainder of the family must contain many more $k$-chains than it ought to.  The relevant calculations are given below.

\begin{PROP} \label{prop:nolargesets}
Suppose $n \ge 15$ and $k \le n - 6$, and let $\cF$ be a set family with $\card{F} = M_k + t$.  If we do not have $1 \le \card{F} \le n-2$ for all $F \in \cF$, then
\[ c_k(\cF) > \binom{n}{a-k} \binom{n-a+k}{k-1} (k-1)! + t \left( \binom{a}{k-1} + \binom{a}{k} \binom{k}{2} \right) (k-1)!. \]
\end{PROP}

\begin{proof}
The same proof as in Lemma \ref{lem:noverylargesets} shows that we may assume we do not have $\emptyset$ or $[n]$ in $\cF$.  Hence it suffices to show there are no sets of size $n-1$ in our family.

\medskip

Suppose towards contradiction we had some set $F \in \cF$ with $\card{F} = n-1$.  We may, by induction, assume that $F$ is in at most $\left(\binom{a}{k-1} + \binom{a}{k} \binom{k}{2} \right) (k-1)!$ $k$-chains.

Since $F$ contains $\binom{n-1}{a-1}$ sets of size $a-1$, there are $\binom{n-1}{a-1} \binom{a-1}{k-2} (k-2)!$ possible $k$-chains that $F$ might be in which consist of $k-1$ sets from the $k-1$ middle levels followed by $F$.  Hence we must be missing a lot of sets from the $k-1$ middle levels to prevent $F$ from being in too many $k$-chains.  The sets of size $a-1$ are contained in the most such $k$-chains, so if we are missing $t_2$ sets, we must have
\[ \left( \binom{n-1}{a-1} - t_2 \right) \binom{a-1}{k-2} (k-2)! \le \left( \binom{a}{k-1} + \binom{a}{k} \binom{k}{2} \right) (k-1)!. \]
Solving for $t_2$ gives
\[ t_2 \ge \binom{n-1}{a-1} - a \left( 1 + \frac{(a-k+1)(k-1)}{2} \right). \]

We now remove from $\cF$ all sets of size $n-1$, thus losing at most $n$ sets, and apply Theorem \ref{thm:allchains} with the above value of $t_2$.  In the theorem, the number of $k$-chains is governed by the expression $\left( t_1 + \left( \frac{k-1}{a} \right) t_2 \right)$.  We are decreasing $t_1$ by at most $n$, but increasing $t_2$ by at least $\binom{n-1}{a-1} - a \left( 1 + \frac{(a-k+1)(k-1)}{2} \right)$, resulting in a net gain in the previous expression of at least
\[ \frac{k-1}{a} \left( \binom{n-1}{a-1} - a \left( 1 + \frac{(a-k+1)(k-1)}{2} \right) \right) - n. \]

Since $n-k \ge 6$, we have $n - a \ge 3$.  If $n - a$ is some constant, then the first term is at least cubic in $n$, while the term we are subtracting is quadratic, since in this case $a - k + 1$ will also be constant.  One the other hand, if $n - a$ is large, the first term will be at least a large power of $n$, while the term we subtract is at most cubic in $n$.  Given $n \ge 15$, some simple but tedious calculations show that in either case, having a set of size $n-1$ increases the number of $k$-chains our family must contain.
\end{proof}

Note that the condition $k \le n - 6$ is near-optimal, since if $k = n - 3$, then $M_k = 2^n - 2n - 2$, and so by volume considerations alone there must be extremal families with sets of size at least $n-1$.

\medskip

Theorem \ref{thm:k+1levels} now follows easily.

\begin{proof}[Proof of Theorem \ref{thm:k+1levels}]
Since $n \ge 15$ and $k \le n-6$, Proposition \ref{prop:nolargesets} shows that $1 \le \card{F} \le n-2$ for all $F \in \cF$.  We may then apply Theorem \ref{thm:allchains} with $t_1 = t$ and $t_2 = 0$ to obtain the bound
\[ C(\cF) \ge \binom{n}{a-k} \binom{n-a+k}{k-1} (k-1)! + t \left( \binom{a}{k-1} + \binom{a}{k} \binom{k}{2} \right) (k-1)!. \]
Recalling that $a = \ceil{\frac{n+k}{2}}$, this is precisely the desired lower bound.

We can again deduce a characterization of the extremal families.  Note that we must have $t_2 = 0$ for the above bound to hold, and so the $k-1$ middle levels must be full.  In order to have equality in Lemma \ref{lem:c2bound}, we also needed $m(F) = a$ for all $F \in \hat{\cF}$.  If $a = \frac{n+k}{2}$, then the remaining $\binom{n}{a} + t$ sets must have size $\frac{n}{2} \pm \frac{k}{2}$.  To obtain equality in \eqref{ineq:LYM3}, every chain must pass through either $k$ or $k+1$ sets of $\cF$, and so we must have $\{G \notin \cF: \card{G} = \frac{n}{2} \pm \frac{k}{2} \}$ forming an antichain.

If, on the other hand, $a = \frac{n+k+1}{2}$, then sets of size $a = \frac{n}{2} + \frac{k+1}{2}$ carry greater weight than sets of size $a-k = \frac{n}{2} - \frac{k-1}{2}$.  We can then redefine $t_2$ above to be the number of sets missing in the $k$ middle levels and obtain the same result.  Hence it follows that we must have all sets in the $k$ middle levels, with the remaining sets in $\hat{\cF}$ of size $\frac{n}{2} \pm \frac{k+1}{2}$.  In order to maintain equality in $\eqref{ineq:LYM3}$, every chain must pass through at most one set $F \in \cF$ with $\card{F} = \frac{n}{2} \pm \frac{k+1}{2}$, and so $\{F \in \cF: \card{F} = \frac{n}{2} \pm \frac{k+1}{2} \}$ must be an antichain.

Thus, once again, the extremal families are exactly the same as those that minimize the number of $2$-chains, as given by Theorem \ref{thm:2chains}.
\end{proof}

\section{Concluding remarks and open problems} \label{sec:conclusion}

In this paper, we have partially answered Kleitman's conjecture by showing that the families that minimize the number of $2$-chains also minimize the number of $k$-chains when they occupy up to the $k+1$ middle levels.  While we strongly believe the conjecture is true in general, we suspect new ideas are needed to deal with larger families.  As the number of levels grows with respect to $k$, the number of different types of chains - in terms of the sizes of the steps between sets - grows rapidly, and these would all need to be controlled to obtain a precise result.  In this direction, though, the same methods we have used above can be applied to show the following: if we have integers $\alpha_1, \alpha_2, \hdots, \alpha_{k-1}$ with $\sum_i \alpha_i = \ell - 1$, then, provided $\card{\cF} \le M_{\ell}$ and the largest set in our family has size at most $n - \max_i \alpha_i$, the number of $k$-chains $F_1 \subset F_2 \subset \hdots \subset F_k$ with $\card{F_{i+1} \setminus F_i} \ge \alpha_i$ is minimized by taking sets in the middle $\ell$ levels.

\medskip

Considering the case of $2$-chains, our paper has focused on showing that a family with more than $\sperner$ sets must contain many $2$-chains.  A closely related problem is to determine whether such a family must have any sets contained in many $2$-chains.  This type of question has been studied before in other settings.  For example, when one is considering the number of triangles in a graph,
Erd\H{o}s showed in \cite{erd62a} that any graph with $\floor{ \frac{n^2}{4} } + 1$ edges must contain an edge in at least $\frac{n}{6} + o(n)$
triangles. It is well-known and easy to see that the hypercube, a graph whose vertices are subsets of $[n]$, with two vertices adjacent if they are comparable and differ in exactly one element, has independence number $2^{n-1}$.  Chung, F\"uredi, Graham and Seymour \cite{chung} proved any induced subgraph on $2^{n-1} + 1$ vertices contains a vertex of degree at least $(\frac12 + o(1)) \log_2 n$.  It is an open
problem to determine whether or not this bound is tight (the corresponding upper bound is $O(\sqrt{n})$), and the answer to this question has ramifications in theoretical computer
science.

In the context of Sperner's theorem the above problem has a negative answer, which may be surprising given the previous two examples. For convenience, let us assume $n = 2m+1$ is odd,
and consider the following set family.  Let
\[ \cF = \{ F : 1 \notin F, |F| = m \} \cup \{ F : 1 \in F, |F| = m + 1 \}. \]
This family contains $2 \binom{2m}{m} = \left( 1 + \frac{1}{n} \right) \sperner$ sets, and so we are indeed beyond the Sperner bound.  However, it is easy to see that the only pairs of comparable sets are of the form $\{ F , \{1\} \cup F \}$ for every $F \in \cF$ with $|F| = m$.  Hence each set of the family is in only one pair of comparable sets.  In fact, for this family we have $c_2(\cF) = c_2(n, |\cF|)$, so it is possible to have an extremal family with the comparable pairs distributed as evenly as possible. Theorem \ref{thm:2chains} shows that any family with $2\binom{2m}{m} + 1$ sets must contain a set in at least two $2$-chains.  It is an open problem as to whether this is also the largest family without a set that contains two other sets (and hence is the maximum set in two $2$-chains).  This configuration is known as a $2$-fork, and the upper bound, which can be obtained using similar arguments as in Theorem \ref{thm:klevels}, is $\left(1 + \frac{2}{n} \right) \sperner$, as shown by Katona and Tarj\'an \cite{kattar}.

\medskip

We find most exciting the prospect of studying Erd\H{o}s-Rademacher-type problems in other settings.  Within the context of Sperner's Theorem, a paper of Qian, Engel and Xu
\cite{qian} studied an extension for multiset families, where the same set may be chosen multiple times.  In a series of two papers, Mubayi \cite{mub10,mub13} extended the Erd\H{o}s-Rademacher results to graphs other than cliques, studying the question for color-critical graphs and some $3$- and $4$-uniform hypergraphs.  In a subsequent paper, we will present Erd\H{o}s-Rademacher type strengthening of the Erd\H{o}s-Ko-Rado Theorem.  However, as one can investigate similar extensions for any extremal result, there is truly no end to the number of directions in which this project can be continued.  We hope that further work of this nature will lead to many interesting results and a greater understanding of classical theorems in extremal combinatorics.

\medskip

\noindent {\bf Note added in proof:} During the preparation of this manuscript, it came to our attention that Dove, Griggs, Kang and Sereni \cite{dove} have independently obtained Theorem \ref{thm:klevels}.

\appendix

\section{The shifting proposition} \label{app:shifting}

In this appendix, we prove Proposition \ref{prop:shifting}, which enables us to perform the shifting necessary for Theorem \ref{thm:2chains}.  As mentioned in Section \ref{sec:2chains}, this is essentially the same shifting argument used in the original proof of Kleitman in \cite{kleitman}.  We provide the proof here as the details of the calculations are not included in Kleitman's paper.

\begin{proof}[Proof of Proposition \ref{prop:shifting}]
Suppose not.  Note that we must have $m \ge 1$, otherwise there is nothing to prove.  Let $\ell \le 2m - 1$ be the minimal integer such that there exists a largest set of size $\frac{n}{2} + m$ with a subset of size $\frac{n}{2} + m - \ell$ that is not in the family.  Let $\cA = \{A \in \cF : |A| = \frac{n}{2} + m, \partial^{\ell} A \not\subset \cF \}$, and let $\cB = \partial^{\ell} \cA \setminus \cF$.  We can construct an auxiliary bipartite inclusion graph on $\cA \cup \cB$, with an edge $(A,B)$ iff $B \subset A$.

Consider first the case where we have a matching $M : \cA \rightarrow \cB$, so that for every set $A \in \cA$ there exists a set $M(A) \subset A$, $M(A) \notin \cF$.  We shift the
family from $\cF$ to $\tilde{\cF}$ by replacing each set $A \in \cA$ by $M(A) \in \cB$, and claim that this reduces the number of $2$-chains.  Note that if $B = M(A)$ is a
newly-introduced set, and $C \in \cF$ is a set with $C \subset B$, then we must have had $C \subset A$ as well.  Thus the only $2$-chains that we need to consider are those
between the levels $\frac{n}{2} + m$ and $\frac{n}{2} + m - \ell$; we call these \emph{intermediate chains}.

Suppose $\ell > 1$.  By the minimality of our choice of $\ell$, we must have $\partial^i \cA \subset \cF$ for every $1 \le i \le \ell - 1$.  Thus the number of intermediate chains in
$\cF$ that we lose is at least $|\cA| \sum_{i = 1}^{\ell - 1} \binom{\frac{n}{2} + m}{i}$.  On the other hand,
all sets of size $\frac{n}{2} + m$ in $\tilde{\cF}$ are the sets from $\cF$ with $i$-shadow completely in $\cF$ for all $i\le \ell$. These sets cannot be involved in any $2$-chains with sets in $\cB$, and therefore we only gain
intermediate chains between the levels $\frac{n}{2} + m-1$ and $\frac{n}{2} + m - \ell$.  The number of
such chains that we gain is at most $|\cA| \sum_{i=1}^{\ell - 1} \binom{\frac{n}{2} - m + \ell}{i}$.  Since $\ell \le 2m - 1$, it follows that $\binom{\frac{n}{2} - m + \ell}{i} <
\binom{\frac{n}{2} + m}{i}$ for every $1 \le i \le \ell - 1$, and hence the number of $2$-chains decreases.

Thus we may assume $\ell = 1$.  If we had $A \in \cA$ and $B \in \cF$ with $B \subset A$, $|B| = \frac{n}{2} + m - 1$, then upon shifting to $\tilde{\cF}$, we lose the $2$-chain
$B \subset A$ and gain no pairs.  Hence we may assume $\partial \cA \cap \cF = \emptyset$, so $\cB = \partial \cA$.  We now claim that the sets $A \in \cA$ cannot be involved in any
$2$-chains $C \subset A$ in $\cF$.  Suppose to the contrary we had such a $2$-chain.  Let $x \in C$ be an arbitrary element of $C$, and shift $A$ to $A \setminus \{ x \}$
(recall that $A \setminus \{ x \}\not \in \cF$).  Shift the remaining sets in $\cA$ by an arbitrary matching from $\cA' = \cA \setminus \{ A \}$ to $\cB' = \partial \cA \setminus \{ A \setminus \{ x \} \}$; we can do this by Hall's Theorem, since every
set in $\cA'$ has at least $\frac{n}{2} + m - 1$ neighbors in $\cB'$, while each set in $\cB'$ has at most $\frac{n}{2} - m + 1$ neighbors.  In this shifted set we have lost the
$2$-chain $C \subset A$, and hence $\tilde{\cF}$ has fewer $2$-chains.

Hence we may assume that there are no $2$-chains in $\cF$ involving sets in $\cA$.  Thus in the shifted family $\tilde{\cF}$, sets in $\partial \cA$ will also not be in any
$2$-chains.  Now, since $|\cF| > \sperner$, it follows from Sperner's Theorem that there is some $2$-chain $C \subset D$ in $\cF$.  In $\tilde{\cF}$,
we may also shift $D$ to some set in $\partial \cA$, since $m \ge 1$ implies $| \partial \cA | > | \cA|$.  As no set in $\partial \cA$ is involved in any $2$-chain, this reduces
the number of $2$-chains, which contradicts the minimality of $\cF$.

Therefore we conclude that there cannot be a matching from $\cA$ to $\cB$ in the auxiliary bipartite inclusion graph, and so we will not shift all sets in $\cA$.  Instead, we use the following lemma, to be proven shortly, to find a collection of sets to shift.

\begin{LEMMA} \label{lem:bipartite}
Let $G$ be a bipartite graph on $U \cup V$ with minimum degree $\delta_U \ge 1$ in $U$ and maximum degree $\Delta_V$ in $V$.  Suppose there is no matching from $U$ to $V$.  Then there exist nonempty subsets $U_1 \subset U$ and $V_1 \subset V$ with a perfect matching $M: U_1 \rightarrow V_1$ and $e(U_1, V) + e( U \setminus U_1, V_1) \le |U_1| \Delta_V$.
\end{LEMMA}

Our auxiliary graph satisfies the conditions of the lemma, with $U=\cA, V=\cB, \Delta_V = \binom{\frac{n}{2} - m + \ell}{\ell} < \binom{\frac{n}{2} + m}{\ell}$, and so we can find a
collection of sets $\cA_1 \subset \cA$ and a matching $M : \cA_1 \rightarrow \cB_1 \subset \cB$ as given by the lemma.  Consider the shifted family $\tilde{\cF}$ where we replace the sets in $\cA_1$
by the corresponding sets in $\cB_1$.  As before, since for every $A \in \cA_1$ we have $M(A) \subset A$, we need only consider the intermediate chains.

Again, by the minimality of $\ell$, we know that $\cA_1$ has full shadow in $\cF$ up until the $\ell$th shadow, and so the same calculation as before implies that we
remove more chains than we gain, and thus have fewer intermediate chains in $\tilde{\cF}$.  Hence it suffices to consider only the new chains formed between levels $\frac{n}{2} + m$ and $\frac{n}{2} + m - \ell$.

The number of new chains between these levels we gain is exactly $e( \cA \setminus \cA_1, \cB_1)$.  On the other hand, we lose all chains between $\cA_1$ and $\partial^{\ell} \cA_1 \cap
\cF = \partial^{\ell} \cA_1 \setminus \cB$.  Thus the number of chains we are losing is $| \cA_1| \binom{ \frac{n}{2} + m }{ \ell } - e(\cA_1, \cB)$.  By the lemma, we have $| \cA_1 |
\binom{ \frac{n}{2} + m }{ \ell } - e( \cA_1, \cB ) > | \cA_1 | \binom{ \frac{n}{2} - m + \ell} { \ell } - e( \cA_1, \cB ) \ge e( \cA \setminus \cA_1, \cB_1 )$, and hence $\tilde{\cF}$
has fewer $2$-chains than $\cF$, contradicting the optimality of $\cF$.

\vspace{0.1in}

Thus if $\cF$ minimizes the number of $2$-chains, and $A$ is the largest set in $\cF$ with $|A| = \frac{n}{2} + m$, then whenever $B \subset A$ with $|B| \ge \frac{n}{2} - m + 1$, we must have $B \in \cF$ as well.

\end{proof}

It remains to furnish a proof of Lemma \ref{lem:bipartite}, which we now provide.

\begin{proof}[Proof of Lemma \ref{lem:bipartite}]
As there is no matching from $U$ to $V$, by Hall's Theorem there exists a minimal subset $U_0 \subseteq U$ with $|N(U_0)| < |U_0|$.  Since $\delta_U \ge 1$, we must have $|U_0| \ge 2$.  Let $u \in U_0$ be an arbitrary element, and take $U_1 = U_0 \setminus \{ u \}$.  By the minimality of $U_0$, it follows that $|N(U_1)| \ge |U_1|$, and so we must have $|N(U_1)| = |U_1|$.  Set $V_1 = N(U_1)$.  Again by the minimality of $U_0$, for any subset $X \subset U_1$, $|N(X)| \ge |X|$, and so by Hall's Theorem there exists a perfect matching $M: U_1 \rightarrow V_1$.

Now $e(U_1, V) + e(U \setminus U_1, V_1) = e(U_1,V) + e(U,V_1) - e(U_1,V_1) = e(U,V_1) \le |V_1| \Delta_V = |U_1| \Delta_V$, where the second equality follows from the fact that $N(U_1) = V_1$, and thus we have the desired inequality.
\end{proof}

\end{document}